\documentclass[12pt]{amsart}

\usepackage{amsthm,amsmath,amsfonts,amssymb}
\usepackage{enumerate}

\allowdisplaybreaks

\theoremstyle{plain}
\newtheorem{thm}{Theorem}[section]
\newtheorem{defn}[thm]{Definition}
\newtheorem{lem}[thm]{Lemma}
\newtheorem{prop}[thm]{Proposition}
\newtheorem{cor}[thm]{Corollary}
\newtheorem{rem}[thm]{Remark}

\title[Convolution structures on Orlicz spaces]{Convolution structures for an Orlicz space with respect to vector measures on a compact group}
\author{Manoj Kumar} 
\address{Department of Mathematics \endgraf Indian Institute of Technology Delhi \endgraf Delhi - 110016 \endgraf India}
\email{manojk9t3@gmail.com}
\author{N. Shravan Kumar} 
\address{Department of Mathematics \endgraf Indian Institute of Technology Delhi \endgraf Delhi - 110016 \endgraf India}
\email{shravankumar@maths.iitd.ac.in}

\begin{document}
\begin{abstract}
The aim of this paper is to present some results about the space $L^\Phi(\nu),$ where $\nu$ is a vector measure on a compact (not necessarily abelian) group and $\Phi$ is a Young function. We show that under certain conditions, the space $L^\Phi(\nu)$ becomes an $L^1(G)$-module with respect to the usual convolution of functions. We also define one more convolution structure on $L^\Phi(\nu).$
\end{abstract}


	
\maketitle

\section{Introduction}
It is well known that, if $G$ is a compact abelian group, then for any $1\leq p\leq \infty,$ the space $L^p(G)$ forms an $L^1(G)$-module. See \cite{F}. In 2009, O. Delgado and P. Miana \cite{DM} generalised this result to an $L^p$ space associated with a vector measure. This note has the modest aim of showing how some of these results can be carried over, not always trivially, to spaces associated to a vector measure on a compact group that is not necessarily abelian. More precisely, we show that an Orlicz space associated with a vector measure on a compact group is an $L^1(G)$-module. 

In the classical case, as is well-known, the Haar measure plays a major role in proving many facts about the $L^p$ spaces. Therefore, in order to prove that an Orlicz space associated with a vector measure is an $L^1(G)$-module, it becomes necessary to assume certain kind of translation invariance of the vector measure.

In Section 3, we show that an Orlicz space associated with a norm integral translation invariant vector measure is a homogeneous space. In Section 4, our main aim is to show that an Orlicz space associated with a vector measure is an $L^1(G)$-module. We also show that the Haar measure is a Rybakov control measure for an absolutely continuous norm integral translation invariant vector measure with some density. In the abelian case, one of the ingredients for the proof of this result is the classical Markov-Kakutani fixed point theorem. This doesn't work if the  group is not abelian. We would like to mention here that the proof given in this paper depends on a fixed point theorem provided in \cite{K}, which works for any compact group.

One of the classical results of abstract harmonic analysis is that an approximate identity in $L^1(G)$ also serves as an approximate identity for the $L^p(G)$ spaces also. Theorem 4.5 of this paper, generalizes this result to the case of an Orlicz space associated with a vector measure.

Finally, in Section 5 we consider another convolution product. This product was introduced in \cite{CFNP} for the $L^p$ spaces associated with a vector measure. In this section, we extend this convolution product to an Orlicz space associated with a vector measure.

Throughout this paper, $G$ will denote a compact group with a fixed normalized Haar measure $m_G.$ 

\section{Vector measures and their associated Orlicz spaces}
Let $G$ be a compact group, $\mathcal{B}(G)$ be the Borel $\sigma$-algebra on $G$ and let $m_G$ be the normalized Haar measure on $G.$ For $1\leq p\leq\infty,$ let $L^p(G)$ denote the usual $p^{th}$ Lebesgue space with respect to the measure $m_G.$ Let $\mathcal{S}(G)$ denote the space of all simple functions on $G$ and $\mathcal{C}(G)$ denote the space of all continuous functions on $G.$

Let $X$ be a complex Banach space and let $\nu$ be a $\sigma$-additive $X$-valued vector measure on $G$. Let $X^\prime$ denote the topological dual of $X$ and let $B_{X^\prime}$ denote the closed unit ball in $X^\prime.$ For each $x^\prime\in X^\prime,$ we shall denote by $\langle\nu,x^\prime\rangle,$ the corresponding scalar valued measure for the vector measure $\nu,$ which is defined as $\langle\nu,x^\prime\rangle(A)=\langle\nu(A),x^\prime\rangle, A\in\mathcal{B}(G).$ A set $A\in\mathcal{B}(G)$ is said to be $\nu$-null if $\nu(B)=0$ for every Borel set $B\subset A.$ The vector measure $\nu$ is said to be {\it absolutely continuous} with respect to a non-negative scalar measure $\mu$ if $\underset{\mu(A)\rightarrow 0}{\lim}\nu(A)=0, A\in\mathcal{B}(G).$ We shall denote this as $\nu\ll\mu.$ The semivariation of $\nu$ on a set $A\in\mathcal{B}(G)$ is defined as $\|\nu\|(A)=\underset{x^\prime\in B_{X^\prime}}\sup |\langle\nu,x^\prime\rangle|(A).$ We shall denote by $\|\nu\|$ the quantity $\|\nu\|(G).$ Let $M_{ac}(G,X)$ denote the space of all $X$-valued measures which are absolutely continuous with respect to the Haar measure $m_G$. A finite positive measure $\mu$ is said to be a {\it Rybakov control measure} for a vector measure $\nu,$ if $\nu\ll\mu=|\langle\nu,x^\prime\rangle|$ for some $x^\prime\in X^\prime.$ It follows from \cite[Theorem IX.2.2]{DU} that such an $x^\prime$ always exists. Further, by \cite[Pg. 10, Theorem 1]{DU} and from the inequality $\mu(A)\leq\|x^\prime\|_{X^\prime}\|\nu\|(A),\ A\in\mathcal{B}(G),$ it follows that the measures $\nu$ and $\mu$ have same null-sets. For more details on vector measures see \cite{DU} and \cite{ORP}.

We say that a map $\Phi:[0,\infty)\rightarrow [0,\infty)$ is a Young function if $\Phi$ is a strictly increasing, convex and continuous such that $\Phi(t)=0$ if and only if $t=0$ and $\underset{t\rightarrow\infty}{\lim}\Phi(t)=\infty.$ The complementary Young function $\Psi$ of the function $\Phi$ is given by $\Psi(t)=\sup\{ts-\Phi(s):s\geq 0\},~t\geq 0.$ For a (finite) positive measure $\mu$ on $G$, let $L^\Phi(\mu)$ denote the Orlicz space, consisting of all complex-valued measurable functions on $G$ with $\|f\|_{L^\Phi(\mu)}<\infty,$ where $\|\cdot\|_{L^\Phi(\mu)}$ is the Luxemberg norm given by $$\|f\|_{L^\Phi(\mu)}=\inf\left\{k>0: \int_G\Phi\left(\frac{|f|}{k}\right)\,d\mu\leq 1\right\}.$$ Note that the Orlicz space is a natural generalisation of the classical $L^p$ spaces. For more on Orlicz spaces see \cite{RR}.

Let $L_w^\Phi(\nu)$ denote the weak Orlicz space with respect to a vector measure $\nu,$ i.e., a space consisting of all complex-valued measurable functions $f$ such that $\|f\|_{\nu,\Phi}<\infty,$ where $\|f\|_{\nu,\Phi}=\underset{x^\prime\in B_{X^\prime}}{\sup}\|f\|_{L^\Phi(|\langle\nu,x^\prime\rangle|)}.$ Let $L^\Phi(\nu)$ denote the closure of simple functions under the norm $\|\cdot\|_{\nu,\Phi}.$ The space $L^\Phi(\nu)$ is known as the Orlicz space with respect to the vector measure $\nu.$ Note that if $\Phi(t)=t,$ then the spaces $L^\Phi(\nu)$ and $L^\Phi_w(\nu)$ coincides with $L^1(\nu)$ and $L^1_w(\nu)$ respectively. We shall denote by $\|\cdot\|_\nu$ the norm on the space $L^1_w(\nu).$ Further, for any Young function $\Phi,$ the spaces $L^\Phi_w(\nu)$ and $L^\Phi(\nu)$ are continuously embedded in $L^1_w(\nu)$ and $L^1(\nu)$ respectively. Moreover, if $f\in L^\Phi_w(\nu),$ then, by \cite[Proposition 4.2]{FF}, $\|f\|_\nu\leq2\|\chi_G\|_{\nu,\Psi}\|f\|_{\nu,\Phi},$ i.e., $L^\Phi_w(\nu)$ is continuously embedded in $L^1(\nu).$ For more details see \cite{D}. 

From now onwards, $X$ will denote a Banach space and $\nu$ an $X$-valued measure on $G.$ Further, $\Phi$ will denote a Young function with $\Psi$ as its complementary Young function.

\section{Homogeneity of the space $L^\Phi(\nu)$}
The main aim of this section is to show that the space $L^\Phi(\nu)$ is a homogeneous space. We begin this section with the definition of a homogeneous space.

A Banach function space $Z$ is said to be a {\it homogeneous space} if for each $f\in Z$ and $t,s\in G,$ we have that
\begin{enumerate}[(i)]
\item $\tau_sf\in Z,$ where $\tau_sf(t)=f(s^{-1}t),$
\item $\|\tau_sf\|_Z=\|f\|_Z$ and 
\item $s\mapsto\tau_sf$ is continuous from $G$ into $Z.$
\end{enumerate}  

Let $h:G\rightarrow G$ be a homeomorphism. For a measurable function $f:G\rightarrow\mathbb{C}$, define a function $f_h:G\rightarrow\mathbb{C}$ by $f_h=f\circ h^{-1}.$ Note that $f_h$ is also a measurable function. Similarly, define a measure $\nu_h$ as $\nu_h(A)=\nu(h(A)),~A\in\mathcal{B}(G).$

We denote by $I_\nu$ the continuous linear operator $I_\nu:L^1(\nu)\rightarrow X$ given by $I_\nu(f)=\int_Gf\,d\nu,\ f\in L^1(\nu)$. See \cite[Pg. 152]{ORP}.
\begin{defn}\mbox{ }
\begin{enumerate}[(i)]
\item A vector measure $\nu$ is said to be a norm integral $h$-invariant vector measure if $\|I_\nu(\phi_h)\|=\|I_\nu(\phi)\|$ $\forall\ \phi\in\mathcal{S}(G).$
\item A Banach function space $Z$ is said to be norm $h$-invariant if for each $f\in Z,$ we have $f_h\in Z$ and $\|f_h\|_Z=\|f\|_Z.$
\end{enumerate}
\end{defn}
Observe that a vector measure $\nu$ is norm integral $h$-invariant if and only if $\nu$ is norm integral $h^{-1}$-invariant. 

Now we prove a lemma which will help us to prove the norm $h$-invariance of (weak) Orlicz spaces with respect to a norm integral $h$-invariant vector measure.
\begin{lem}\label{dual}Let $\nu$ be a norm integral $h$-invariant vector measure and $x^\prime\in X^\prime$. Then there exists $x_h^\prime\in X^\prime$ such that $\|x_h^\prime\|\leq\|x^\prime\|.$ Further, $\langle\nu_h,x^\prime\rangle(A)=\langle\nu,x_h^\prime\rangle(A),~A
\in\mathcal{B}(G).$
\end{lem}
\begin{proof}
Define a map $T_h:I_\nu(\mathcal{S}(G))\subset X\rightarrow I_\nu(\mathcal{S}(G))$ by $T_h(I_\nu(\phi))=I_\nu(\phi_h).$ Since $\nu$ is norm integral $h$-invariant, it follows that the map $T_h$ is an isometry. Now, for $x^\prime\in X^\prime,$ let $y_h^\prime=x^\prime\circ T_h.$ Then $y_h^\prime\in (I_\nu(\mathcal{S}(G)))^\prime$ and $\|y_h^\prime\|\leq\|x^\prime\|.$ Further, by Hahn-Banach extension theorem, there exists $x_h^\prime\in X^\prime$ such that $x_h^\prime=y_h^\prime$ on $I_\nu(\mathcal{S}(G))$ and $\|x_h^\prime\|=\|y_h^\prime\|\leq\|x^\prime\|.$ 

Now, let $A\in\mathcal{B}(G).$ Then, we have, 
\begin{align*}
\langle\nu_h,x^\prime\rangle(A)=&\langle\nu_h(A),x^\prime\rangle = \langle I_\nu((\chi_A)_h),x^\prime\rangle\\ =&\langle T_h(I_\nu((\chi_A)),x^\prime\rangle=\langle I_\nu((\chi_A),x^\prime\circ T_h\rangle \\ =&\langle\nu(A),y_h^\prime\rangle=\langle\nu,x_h^\prime
\rangle(A).\qedhere
\end{align*}
\end{proof}
We now prove that the spaces $L^\Phi_w(\nu)$ and $L^\Phi(\nu)$ are norm $h$-invariant for a norm integral $h$-invariant vector measure $\nu$.
\begin{thm}\label{2T1}Let $\nu$ be a norm integral $h$-invariant vector measure. Then the spaces $L^\Phi_w(\nu)$ and $L^\Phi(\nu)$ are norm $h$-invariant.
\end{thm}
\begin{proof}
Since $\nu$ is a norm integral $h$-invariant vector measure, for any $x^\prime\in X^\prime,$ by Lemma \ref{dual}, there exists $x_h^\prime\in X^\prime$ such that $\|x_h^\prime\|\leq\|x^\prime\|$ and $\langle\nu_h,x^\prime\rangle(A)=\langle\nu,x_h^\prime\rangle(A),\forall\ A
\in\mathcal{B}(G).$ We now claim that $\|f_h\|_{L^\Phi(|\langle\nu,x^\prime\rangle|)}=\|f\|_{L^\Phi(|\langle\nu,x_h^\prime\rangle|)},\ \forall\ f\in L^\Phi_w(\nu).$ Let $f=\sum_{i=1}^n\alpha_i\chi_{A_i}\in\mathcal{S}(G)$. Then, for $k>0,$ using Lemma \ref{dual} and the fact that the sets $A_i$'s are disjoint, we have 
\begin{align*}
\int_G\Phi\left(\frac{|f_h|}{k}\right)\,d|\langle\nu,x^\prime\rangle|=& \int_G\Phi\left(\sum_{i=1}^n\frac{|\alpha_i|}{k}\chi_{h(A_i)}\right)\,d|\langle\nu,x^\prime\rangle|\\=&\sum_{i=1}^n\Phi\left(\frac{|\alpha_i|}{k}\right)|\langle\nu,x^\prime\rangle|(h(A_i))\\=&\sum_{i=1}^n\Phi\left(\frac{|\alpha_i|}{k}\right)|\langle\nu,x_h^\prime\rangle|(A_i)\\=&\int_G\Phi\left(\frac{|f|}{k}\right)\,d|\langle\nu,x_h^\prime\rangle|.
\end{align*}
Now, the claim follows from monotone convergence theorem. 

Let $f\in L^\Phi_w(\nu).$ Then, using Lemma \ref{dual}, we have,
\begin{align*}
\|f_h\|_{\nu,\Phi} =&\underset{x^\prime\in B_{X^\prime}}{\sup}\|f_h\|_{L^\Phi(|\langle\nu,x^\prime\rangle|)} \\ =&\underset{x^\prime\in B_{X^\prime}}{\sup}\|f\|_{L^\Phi(|\langle\nu,x_h^\prime\rangle|)} \\ \leq &\underset{x_h^\prime\in B_{X^\prime}}{\sup}\|f\|_{L^\Phi(|\langle\nu,x_h^\prime\rangle|)}=\|f\|_{\nu,\Phi}.
\end{align*} 
Note that if $\nu$ is a norm integral $h$-invariant measure, then $\nu$ is a norm integral $h^{-1}$-invariant measure and therefore $\|f\|_{\nu,\Phi}=\|(f_h)_{h^{-1}}\|_{\nu,\Phi}\leq\|f_h\|_{\nu,\Phi}.$ Hence the space $L^\Phi_w(\nu)$ is norm $h$-invariant.

Now let $f\in L^\Phi(\nu),$ then there exists a sequence $(\phi_n)$ of simple functions such that $\phi_n\rightarrow f$ in $\|\cdot\|_{\nu,\Phi}$ norm. Since $(\phi_n)_h\subset\mathcal{S}(G)$ and $$\|(\phi_n)_h-f_h\|_{\nu,\Phi}=\|(\phi_n-f)_h\|_{\nu,\Phi}=\|\phi_n-f\|_{\nu,\Phi},$$ implies that $f_h\in L^\Phi(\nu)$ as $L^\Phi(\nu)$ is the closure of simple functions. Hence it follows that $L^\Phi(\nu)$ is norm $h$-invariant.
\end{proof}
Next we show that the space of continuous functions on $G$ is dense in $L^\Phi(\nu).$ This will be used to prove homogeneity of the Orlicz space with respect to a norm integral translation invariant vector measure.
\begin{prop}\label{density}
Let $\nu\in M_{ac}(G,X).$ Then $\mathcal{C}(G)$ is dense in $L^\Phi(\nu).$
\end{prop}
\begin{proof}
Let $\phi\in\mathcal{S}(G)$ and $\epsilon>0.$ For every $\delta>0,$ using Lusin's theorem, there exists $f\in \mathcal{C}(G)$ such that $\phi=f$ on $A^c$ and $\|f\|_\infty\leq\|\phi\|_\infty$ where $A^c$ is the complement of the set $A\in\mathcal{B}(G)$ with $m_G(A)<\delta.$ By \cite[p. 78, Corollary 7]{RR}, we have  
\begin{align*}
\|\phi-f\|_{\nu,\Phi}\leq&\|\phi-f\|_\infty\|\chi_A\|_{\nu,\Phi} \\ \leq & 2\|\phi\|_\infty\underset{x^\prime\in B_{X^\prime}}{\sup}\left[\Phi^{-1}\left(\frac{1}{|\langle\nu,x^\prime\rangle|(A)}\right)\right]^{-1}.
\end{align*} 
Since $\nu\in M_{ac}(G,X),$ by choosing $\delta$ small enough, we get $\|\phi-f\|_{\nu,\Phi}<\epsilon.$  Now the result follows from the density of $\mathcal{S}(G)$ in $L^\Phi(\nu).$
\end{proof}
\begin{rem}
Since the trigonometric polynomials on $G$ are dense in $\mathcal{C}(G)$ with respect to the uniform norm and it follows from Proposition \ref{density} that the trigonometric polynomials on $G$ are dense in $L^\Phi(\nu)$ for $\nu\in M_{ac}(G,X).$
\end{rem}
Now we conclude the section with the main result that an Orlicz space with respect to a norm integral translation invariant vector measure is homogeneous. 
\begin{thm}\label{homogeneity}
Let $\nu\in M_{ac}(G,X)$ be norm integral translation invariant. Then the space $L^\Phi(\nu)$ is homogeneous.
\end{thm}
\begin{proof}
Let $f\in L^\Phi(\nu).$ By Theorem \ref{2T1}, $L^\Phi(\nu)$ is norm translation invariant and it is enough to show that the map $s\rightarrow\tau_sf$ is continuous from $G$ to $L^\Phi(\nu).$ Let $\epsilon>0.$ By Proposition \ref{density}, there exists $g\in \mathcal{C}(G)$ such that $\|f-g\|_{\nu,\Phi}<\epsilon/3.$ Further, for $t,s\in G,$ 
\begin{align*} \|\tau_tf-\tau_sf\|_{\nu,\Phi}\leq&\|\tau_tf-\tau_tg\|_{\nu,\Phi}+\|\tau_tg-\tau_sg\|_{\nu,\Phi}+\|\tau_sg-\tau_sf\|_{\nu,\Phi}\\\leq& 2\|f-g\|_{\nu,\Phi}+\|\tau_tg-\tau_sg\|_\infty\|\chi_G\|_{\nu,\Phi}\\<&\frac{2\epsilon}{3}+\|\tau_tg-\tau_sg\|_\infty\|\chi_G\|_{\nu,\Phi}.
\end{align*}
As $G$ is compact, $g$ is uniformly continuous and therefore there exists a neighbourhood $N$ of $e$ in $G$ such that $|g(t)-g(s)|<\frac{\epsilon}{3}\|\chi_G\|_{\nu,\Phi}^{-1},$ for every $t,s\in G$ such that $ts^{-1}\in N.$ Hence we have $\|\tau_tf-\tau_sf\|_{\nu,\Phi}<\epsilon$ for every $t,s\in G$ such that $ts^{-1}\in N.$
\end{proof}

\section{Convolution product for $L^\Phi_w(\nu)$ and $L^\Phi(\nu)$}
Throughout this section, $\nu$ will denote an absolutely continuous norm integral translation invariant vector measure. In this section, we first show that the Haar measure $m_G$ is a Rybakov control measure for the vector measure $\nu$ with some density. We will then show that $L^\Phi_w(\nu)$ and $L^\Phi(\nu)$ are Banach algebras under some conditions on $\nu.$

Note that by \cite[Pg. 108, Theorem 3.7]{ORP}, the space $L^1(\nu)$ is an order continuous Banach function space w.r.t. any Rybakov control measure for $\nu$. Hence by \cite[Pg. 35, Proposition 2.16]{ORP} the K\"othe dual $L^1(\nu)^*$ of $L^1(\nu)$ coincides with the topological dual of $L^1(\nu).$ For more details on K\"othe dual see \cite{ORP}.

\begin{thm}\label{2T2}If $\mu$ is a Rybakov control measure for $\nu$, then there exists $0< h\in L^1(\nu)^*$ such that $\|h\|_{L^1(\nu)^*}=\|\nu\|^{-1}$, where $L^1(\nu)^*$ is the K\"othe dual space of $L^1(\nu).$ Further $m_G(A)=\int_Ah\,d\mu,~A\in\mathcal{B}(G).$
\end{thm}
\begin{proof}
Consider the space $L^1(\nu)^\prime$ with the weak* topology. Consider the family $(\tau_t^\ast)_{t\in G}$ of linear operators, where  $\tau_t^\ast:L^1(\nu)^\prime\rightarrow L^1(\nu)^\prime$ is given by $\tau_t^*(f^\prime)=f^\prime\circ\tau_t,~f^\prime\in L^1(\nu)^\prime.$ Taking translation as homeomorphism and $\Phi(t)=t$ in Theorem \ref{2T1}, we have that $\tau_t:L^1(\nu)\rightarrow L^1(\nu)$ is an isometric isomorphism, it follows that $\tau_t^\ast$ is a weak*-weak* continuous linear operator. Further, note that the map $t\mapsto\tau_t^\ast$ is a continuous map from the compact group $G$ into $\mathcal{B}(L^1(\nu)^\prime)$ and hence the set $\{\tau_t^\ast:t\in G\}$ is compact. In particular, it is a totally bounded group. 

Now, consider the set $S=\{f^\prime\in B_{L^1(\nu)^\prime}:f^\prime(\chi_G)=\|\nu\|\}.$ Since $\nu$ is non-zero,  it is clear that $\|\chi_G\|_\nu=\|\nu\|\neq 0$ and therefore $0\neq\chi_G\in L^1(\nu).$ Hence, by Hahn-Banach theorem, there exists $f^\prime\in L^1(\nu)^\prime$ such that $\|f^\prime\|_{L^1(\nu)^\prime}=1$ and $f^\prime(\chi_G)=\|\chi_G\|_\nu=\|\nu\|.$ This implies that $S$ is non-empty. 

Let $f_1^\prime,f_2^\prime\in S$ and $r\in(0,1).$ Then 
\begin{align*}
\|rf_1^\prime+(1-r)f_2^\prime\|_{L^1(\nu)^\prime}=&\underset{f\in B_{L^1(\nu)}}{\sup}|(rf_1^\prime+(1-r)f_2^\prime)(f)|\\\leq&\underset{f\in B_{L^1(\nu)}}{\sup}(r\|f_1^\prime\|_{L^1(\nu)^\prime}+(1-r)\|f_2^\prime\|_{L^1(\nu)^\prime})\|f\|_\nu\\\leq&1
\end{align*}
and $(rf_1^\prime+(1-r)f_2^\prime)(\chi_G)=rf_1^\prime(\chi_G)+(1-r)f_2^\prime(\chi_G)=\|\nu\|.$ Thus, $S$ is convex. 

Define $\beta:L^1(\nu)^\prime\rightarrow\mathbb{C}$ by $\beta(f^\prime)=f^\prime(\chi_G).$ It is clear that $\beta$ is continuous. Then the set $S=\beta^{-1}(\{\|\nu\|\})\cap B_{L^1(\nu)^\prime}$ is a closed subset of $B_{L^1(\nu)^\prime}$ and $B_{L^1(\nu)\prime}$ is compact by Banach-Alaoglu theorem. Therefore, $S$ is compact. 

Let $t\in G$ and $f^\prime\in S.$ Then, $$\|\tau_t^\ast(f^\prime)\|_{L^1(\nu)^\prime}\leq\|\tau_t^\ast\|\|f^\prime\|_{L^1(\nu)^\prime}=\|\tau_t\|\|f^\prime\|_{L^1(\nu)^\prime}\leq 1$$ which implies that $\tau_t^\ast(f^\prime)\in B_{L^1(\nu)^\prime}.$ Further, $\tau_t^\ast(f^\prime)(\chi_G)=f^\prime(\tau_t\chi_G)=f^\prime(\chi_G)=\|\nu\|.$ Thus, $\tau_t^\ast(f^\prime)\in S.$ Hence, $\tau_t^\ast$ maps the non-empty, convex and compact set $S$ univalently onto itself for every $t\in G$. Therefore, by \cite[Corollary of Theorem 2, p. 245]{K}, there exists a common fixed point $f_0^\prime\in S$ for the family $(\tau_t^\ast)_{t\in G}$, that is, $\tau_t^\ast(f_0^\prime)=f_0^\prime~\forall~t\in G.$ 

Now, by the definition of the set $S$ and from the definition of the norm, it follows that $\|f_0^\prime\|_{L^1(\nu)^\prime}=1.$ Further, by \cite[Pg. 35, Proposition 2.16]{ORP}, there exists $f_0^\ast\in L^1(\nu)^\ast$ such that $\|f_0^\ast\|_{L^1(\nu)^\ast}=\|f_0^\prime\|_{L^1(\nu)^\prime}=1$ and $$f_0^\prime(f)=\int_G ff_0^\ast\ d\mu\ \forall\ f\in L^1(\nu).$$ Since $\chi_G\in L^1(\nu),$ it follows that $f_0^\ast\in L^1(\mu).$ 

Define a scalar measure $\mu_{f_0^\ast}$ as $$\mu_{f_0^\ast}(A)=\int_A f_0^\ast\ d\mu,\ \ A\in\mathcal{B}(G).$$ As $f_0^\prime$ is a fixed point for the family $\{\tau_t^\ast:t\in G\},$ it follows that the measure $\mu_{f_0^\ast}$ is translation invariant and hence so is the measure $|\mu_{f_0^\ast}|.$ Further, the density of $|\mu_{f_0^\ast}|$ is $|f_0^\ast|$ and hence non-zero. We now claim that the measure $|\mu_{f_0^\ast}|$ is regular. Let $A\in\mathcal{B}(G).$ Then $$|\mu_{f_0^\ast}|(A)=\int_A |f_0^\ast|\ d\mu\leq\|f_0^\ast\|_{L^1(\nu)^\ast}\|\chi_A\|_{L^1(\nu)}=\|\nu\|(A).$$ It now follows, from our assumption that $\nu\ll m_G,$ that $|\mu_{f_0^\ast}|\ll m_G$ and therefore by Radon-Nikodym theorem it follows that $|\mu_{f_0^\ast}|$ is regular. In particular, $|\mu_{f_0^\ast}|$ is a Haar measure and hence there exists a positive constant $c$ such that $|\mu_{f_0^\ast}|=cm_G.$ Note that $c=|\mu_{f_0^\ast}|(G)$ as $m_G$ is normalized. As $|f_0^\ast|$ is the density for the measure $|\mu_{f_0^\ast}|,$ it follows that the density for the Haar measure $m_G$ is $\frac{1}{|\mu_{f_0^\ast}|(G)}|f_0^\ast|.$

Let $h=\frac{1}{|\mu_{f_0^\ast}|(G)}|f_0^\ast|.$ Using the fact that $f_0^*\in L^1(\nu)^\ast,$ it follows that $h\in L^1(\nu)^*.$ Further, $$\|\nu\|=f_0^\prime(\chi_G)=\int_G f_0^\ast\ d\mu=\left|\int_G f_0^\ast\ d\mu\right|\leq \int_G |f_0^\ast|\ d\mu=|\mu_{f_0^\ast}|(G).$$ Thus $\|h\|_{L^1(\nu)^\ast}=\|\nu\|^{-1}.$
\end{proof}
\begin{rem}
Theorem $\ref{2T2}$ implies that the measure $m_G$ is a Rybakov control measure with some density and therefore $\nu$-null sets coincides with $m_G$-null sets.
\end{rem}
\begin{cor}\label{2T4}
The embedding of $L^1_w(\nu)$ inside $L^1(G)$ is continuous. Further $\|f\|_1\leq\frac{1}{\|\nu\|}\|f\|_\nu,~f\in L^1_w(\nu).$
\end{cor}
\begin{proof}
Let $\mu$ be a Rybakov control measure for $\nu.$ Since $\nu\ll m_G,$ by Theorem \ref{2T2}, there exists a function $0< h\in L^1(\nu)^\ast$ such that the Haar measure $m_G$ is given by $$m_G(A)=\int_A h\ d\mu,\ A\in\mathcal{B}(G)$$ and $\|h\|_{L^1(\nu)^\ast}=\frac{1}{\|\nu\|}.$ Now, let $f\in L^1_w(\nu).$ By monotone convergence theorem, it suffices to assume that $f$ is a simple function. Then,
\begin{align*}
\|f\|_1=\int_G |f|\ dm_G=&\int_G h|f|\ d\mu\leq \|h\|_{L^1(\nu)^\ast}\|f\|_{\nu}=\frac{1}{\|\nu\|}\|f\|_\nu.\qedhere
\end{align*}
\end{proof}
Since $L^\Phi_w(\nu)\hookrightarrow L^1_w(\nu),$ we have that $L^\Phi_w(\nu)\hookrightarrow L^1(G).$ Therefore, we can consider the classical convolution of any two functions in $L^\Phi_w(\nu).$ In the next theorem, we prove that the (weak) Orlicz space with respect to $\nu$ is an $L^1(G)$-module for the classical convolution.
\begin{thm}\label{2T3}
Let $f\in L^1(G)$ and $g\in L^\Phi_w(\nu).$ Then $f*g\in L^\Phi_w(\nu)$ with $\|f*g\|_{\nu,\Phi}\leq\|f\|_1\|g\|_{\nu,\Phi}$. In particular, if $g\in L^\Phi(\nu)$ then $f*g\in L^\Phi(\nu).$
\end{thm}
\begin{proof}Let $f\in L^1(G)$ and $g\in L^\Phi_w(\nu).$ Using Theorem \ref{2T1}, with translation as a homeomorphism, we have that $\tau_sg\in L^\Phi_w(\nu)$ with $\|\tau_sg\|_{\nu,\Phi}=\|g\|_{\nu,\Phi}.$ Then 
\begin{align*}
\|f*g\|_{\nu,\Phi}=&\left\|\int_Gf(s)\tau_sg\,dm_G(s)\right\|_{\nu,\Phi}\\\leq&\int_G|f(s)|\|\tau_sg\|_{\nu,\Phi}\,dm_G(s)
\\=&\int_G|f(s)|\|g\|_{\nu,\Phi}\,dm_G(s)=\|f\|_1\|g\|_{\nu,\Phi}.
\end{align*}
Now, let $g\in L^\Phi(\nu),$ then there exists two sequences $(\phi_n)$ and $(\psi_n)$ of simple functions converging to $f$ in $L^1(G)$ and $g$ in $L^\Phi(\nu)$ respectively. Since, for each $n\in\mathbb{N},$ $\phi_n\in L^1(G)$ and $\psi_n\in L^\infty(G),$ we have that $\phi_n*\psi_n$ is bounded and therefore $\phi_n*\psi_n\in L^\Phi(\nu).$ Further,
\begin{align*}
\|\phi_n*\psi_n-f*g\|_{\nu,\Phi}\leq&\|(\phi_n-f)*\psi_n\|_{\nu,\Phi}+\|f*(\psi_n-g)\|_{\nu,\Phi}\\\leq&\|\phi_n-f\|_1\|\psi_n\|_{\nu,\Phi}+\|f\|_1\|\psi_n-g\|_{\nu,\Phi}.
\end{align*}
Therefore, the sequence $(\phi_n*\psi_n)$ converges to $f*g$ in the $\|\cdot\|_{\nu,\Phi}$ norm. Using the fact that $L^\Phi(\nu)$ is a closed subspace of $L^\Phi_w(\nu),$ we have that $f*g\in L^\Phi(\nu).$ 
\end{proof}

Our next result is the analogue of \cite[Proposition 2.42]{F}.
\begin{thm}[Existence of left approximate identity]\label{approximateidentity}The space $L^\Phi(\nu)$ has a left approximate identity, that is, there exists a net $(g_\alpha)_{\alpha\in\wedge}$ in $L^\Phi(\nu)$ such that  
\begin{enumerate}[(i)]
\item $(g_\alpha)_{\alpha\in\wedge}\subset \mathcal{C}(G),$
\item $g_\lambda\geq 0\ \forall\ \lambda\in\wedge,$
\item $supp (g_\beta)\subset supp (g_\alpha)\mbox{ if }\alpha\preccurlyeq\beta,\ \alpha,\beta\in\wedge,$
\item $supp(g_\lambda)\rightarrow\{e\}$ as $\lambda\rightarrow\infty,$
\item $\int_Gg_\lambda\,dm_G=1\ \forall\ \lambda\in\wedge$ and
\item $\underset{\lambda}{\lim}\|g_\lambda*f-f\|_{\nu,\Phi}=0,~f\in L^\Phi(\nu).$
\end{enumerate}   
\end{thm}
\begin{proof}
It is well-known that a net $(g_\alpha)_{\alpha\in\wedge}$ satisfying (i) to (v) exists in $L^1(G).$ This net satisfies our requirements. Indeed, for $\lambda\in\wedge$ and $f\in\ L^\Phi(\nu),$
\begin{align*}
\|g_\lambda*f-f\|_{\nu,\Phi}=&\left\|\int_Gg_\lambda(s)\tau_sf\,dm_G(s)-f\int_Gg_\lambda(s)\,dm_G(s)\right\|_{\nu,\Phi}\\\leq&\int_Gg_\lambda(s)\|\tau_sf-f\|_{\nu,\Phi}\,dm_G(s).
\end{align*}
Taking $\lambda\rightarrow\infty$ and using Theorem \ref{homogeneity}, the result follows.
\end{proof}
The next theorem is an analogue of \cite[Theorem 2.43]{F}.
\begin{thm}Let $I$ be a closed subspace of $L^\Phi(\nu).$ Then $I$ is  a left $L^1(G) $-submodule of $L^\Phi(\nu)$ if and only if $I$ is a translation invariant subspace.
\end{thm}
\begin{proof}
Suppose that $I$ is a left $L^1(G)$-submodule of $L^\Phi(\nu).$ Let $f\in I$ and $t\in G.$ Let $(g_\lambda)_{\lambda\in\wedge}$ be a left approximate identity of $L^\Phi(\nu),$ provided by Theorem \ref{approximateidentity}. Then, using Theorem \ref{2T1}, $$\|(\tau_tg_\lambda)*f-\tau_tf\|_{\nu,\Phi}=\|\tau_t(g_\lambda*f-f)\|_{\nu,\Phi}=\|g_\lambda*f-f\|_{\nu,\Phi}\rightarrow 0.$$ 
Since $(\tau_tg_\lambda)*f\in I$ and as $I$ is a closed subspace of $L^\Phi(\nu)$, it follows that $\tau_tf\in I.$

Now let $I$ be a translation invariant space. Let $f\in L^1(G)$ and $g\in I.$ Since $f*g=\int_Gf(s)\tau_sg\,dm_G(s)$, it is clear that $f*g$ belongs to the closed linear span of the functions from $I.$ Since $I$ is a closed subspace of $L^\Phi(\nu)$, it follows that $f*g\in I.$
\end{proof}

Now we discuss the convolution structure in the spaces $L^\Phi_w(\nu)$ and $L^\Phi(\nu)$.
\begin{thm}Let $f,g\in L^\Phi_w(\nu).$ Then $f*g\in L^\Phi_w(\nu)$ with $$\|f*g\|_{\nu,\Phi}\leq\frac{2}{\|\nu\|}\|\chi_G\|_{\nu,\Psi}\|f\|_{\nu,\Phi}\|g\|_{\nu,\Phi}.$$
In particular, if $f,g\in L^\Phi(\nu)$ then $f*g\in L^\Phi(\nu).$
\end{thm}
\begin{proof}
If $f,g\in L^\Phi_w(\nu)$ then using Theorem \ref{2T3} and the fact that $L^\Phi_w(\nu)\hookrightarrow L^1(G),$ we have $f*g\in L^\Phi_w(\nu)$ with $\|f*g\|_{\nu,\Phi}\leq\|f\|_1\|g\|_{\nu,\Phi}.$  Further, using \cite[Proposition 4.2]{FF} and Corollary \ref{2T4}, we have 
\begin{align*}
\|f*g\|_{\nu,\Phi}\leq&\|f\|_1\|g\|_{\nu,\Phi}\leq\frac{1}{\|\nu\|}\|f\|_\nu\|g\|_{\nu,\Phi}\\\leq&\frac{2}{\|\nu\|}\|\chi_G\|_{\nu,\Psi}\|f\|_{\nu,\Phi}\|g\|_{\nu,\Phi}.
\end{align*}
The conclusion for $L^\Phi(\nu)$ follows from the density of $\mathcal{S}(G)$ in $L^\Phi(\nu).$
\end{proof}

We finish this section with a remark on the Banach algebra structure of the (weak) Orlicz spaces with respect to $\nu.$
\begin{rem}If the vector measure $\nu$ satisfies $\|\nu\|\geq 2\|\chi_G\|_{\nu,\Psi},$ then the spaces $L^\Phi_w(\nu)$ and $L^\Phi(\nu)$  become Banach algebras with classical convolution as multiplication.
\end{rem}

\section{Another convolution product for $L^\Phi_w(\nu)$ and $L^\Phi(\nu)$}
In this final section, we define another convolution product for (weak) Orlicz spaces with respect to the vector measure $\nu$. The main result of this section is Theorem 5.2.

If $\nu\ll m_G,$ then $\langle\nu,x^\prime\rangle\ll m_G,~\forall\ x^\prime\in X^\prime.$ Therefore, by Radon-Nikodym theorem there exists a function $h_{x^\prime}\in L^1(G)$ such that $d\langle\nu,x^\prime\rangle=h_{x^\prime}\,dm_G.$ Further, for $f\in L^1_w(\nu)$, $$\int_G|f|(t)d|\langle\nu,x^\prime\rangle|(t)=\int_G|fh_{x^\prime}|(t)\,dm_G(t).$$ Hence, $fh_{x^\prime}\in L^1(G)$ for every $f\in L^1_w(\nu).$ With this as motivation, we define another convolution product in the spirit of \cite[Definition 4.1]{CFNP} and study some inequalities.
\begin{defn}\label{conv2}
Let $1\leq p\leq\infty.$ The convolution of functions $f\in L^1_w(\nu)$ and $g\in L^p(G)$ with respect to a vector measure $\nu\in M_{ac}(G,X)$ is defined by $$f\mathbf{*}_{\nu}g(x^\prime)=(fh_{x^\prime})*g,~x^\prime\in X^\prime.$$
\end{defn}
Note that $f \mathbf{*}_{\nu}g(x^\prime)(t)=\int_Gf(s)\tau_sg(t)\,d\langle
\nu,x^\prime\rangle(s),~\forall\ t\in G.$ Since for a norm integral translation invariant vector measure $\nu,$ $L^\Phi_w(\nu)$ is continuously embedded in $L^1(G),$  Definition \ref{conv2} makes sense for functions $f\in L^1_w(\nu)$ and $g\in L^\Phi_w(\nu)$. With this note, here is the main result of this section.
\begin{thm}\label{convo2ineq}
Let $\nu\in M_{ac}(G,X)$ be a norm integral translation invariant vector measure. If $f\in L^1_w(\nu)$ and $g\in L^\Phi_w(\nu),$ then $f\mathbf{*}_{\nu}g\in \mathcal{B}(X^\prime, L^\Phi_w(\nu))$ and $$\|f\mathbf{*}_{\nu}g\|_{\mathcal{B}(X^\prime, L^\Phi_w(\nu))}\leq\|f\|_\nu\|g\|_{\nu,\Phi}.$$ In particular, if $f\in L^1(\nu)$ and $g\in L^\Phi(\nu)$ then $f\mathbf{*}_{\nu}g\in \mathcal{B}(X^\prime, L^\Phi(\nu)).$
\end{thm}
\begin{proof}
Let $x^\prime\in B_{X^\prime}.$ Using Theorem \ref{2T1}, we have,
\begin{align*}
\|f\mathbf{*}_{\nu}g(x^\prime)\|_{\nu,\Phi}=&\left\|\int_Gf(s)\tau_sg\,d\langle
\nu,x^\prime\rangle(s)\right\|_{\nu,\Phi}
\\\leq&\int_G|f(s)|\|\tau_sg\|_{\nu,\Phi}\,d|\langle
\nu,x^\prime\rangle|(s)\\\leq&\int_G|f(s)|\|g\|_{\nu,\Phi}\,d|\langle
\nu,x^\prime\rangle|(s)\leq\|f\|_\nu\|g\|_{\nu,\Phi}.
\end{align*}
Thus, it follows that, $f\mathbf{*}_{\nu}g\in \mathcal{B}(X^\prime, L^\Phi_w(\nu))$ with $\|f\mathbf{*}_{\nu}g\|_{\mathcal{B}(X^\prime, L^\Phi_w(\nu))}\leq\|f\|_\nu\|g\|_{\nu,\Phi}.$ 

The last statement follows again from the density of the space of simple functions.
\end{proof}
Using Theorem \ref{convo2ineq} and \cite[Proposition 4.2]{FF} we have the following immediate consequence.
\begin{cor}
Let $\nu\in M_{ac}(G,X)$ be a norm integral translation invariant vector measure. If $f,g\in L^\Phi_w(\nu)$ then $f\mathbf{*}_{\nu}g\in \mathcal{B}(X^\prime, L^\Phi_w(\nu))$ and $\|f\mathbf{*}_{\nu}g\|_{\mathcal{B}(X^\prime, L^\Phi_w(\nu))}\leq2\|\chi_G\|_{\nu,\Psi}\|f\|_{\nu,\Phi}\|g\|_{\nu,\Phi}.$ In particular, if $f,g\in L^\Phi(\nu)$ then $f\mathbf{*}_{\nu}g\in \mathcal{B}(X^\prime, L^\Phi(\nu)).$
\end{cor}

\section*{Acknowledgment}
The first author would like to thank the University Grants Commission, India for providing the research grant. 

\end{document}